\documentclass[10pt,twoside]{article}

\usepackage{amsmath}
\usepackage{amssymb}
\usepackage{amsthm}
\usepackage[latin1]{inputenc}
\usepackage{eurosym}
\usepackage[dvips]{graphics}
\usepackage{graphicx}
\usepackage{epsfig}
\usepackage{hyperref}
\usepackage{ifthen}
\usepackage{setspace}
\usepackage[active]{srcltx}
\usepackage{pdfsync}
\usepackage{color}

\newcommand{\Rr}{{\mathbb{R}}}

\newcommand{\D}[1]{\mbox{\rm #1}}
\newcommand{\dd}{\D{d}}

\newtheorem{theorem}{Theorem}[section]

\newtheorem{rem}[theorem]{\sc Remark}
    \newenvironment{remark}{\begin{rem} \begin{rm}}{\end{rm} \qed\end{rem}}
\newtheorem{lemma}[theorem]{Lemma}

 \newtheorem{proposition}[theorem]{Proposition}

\newtheorem{definition}[theorem]{Definition}

\theoremstyle{definition}

\numberwithin{equation}{section}

\makeatletter
\evensidemargin \oddsidemargin
\makeatother

\begin{document}
\thispagestyle{empty}
\setcounter{page}{1}

%\noindent
%{\footnotesize {\rm To appear in\\[-1.00mm]
%{\em Dynamics of Continuous, Discrete and Impulsive Systems}}\\[-1.00mm]
%http:monotone.uwaterloo.ca/$\sim$journal} $~$ \\ [.3in]

%%USE THE STUFF BELOW AS A GUIDE TO SET UP THE START OF PAPER%%

\begin{center}
{\large\bf \uppercase{On a fractional sublinear elliptic equation with a variable coefficient}}

\vskip.20in

Fabio Punzo \\[2mm]
{\footnotesize
Dipartimento di Matematica "G. Castelnuovo", \\
Universit\`{a} di Roma La Sapienza,\\
Piazzale Aldo Moro 5, 00185 Roma \\
punzo@mat.uniroma1.it}\\[2mm]
Gabriele Terrone\footnote{
Gabriele Terrone was supported by the UTAustin-Portugal partnership through the FCT post-doctoral fellowship
SFRH/BPD/40338/2007, CAMGSD-LARSys through FCT Program POCTI -
FEDER and by grants PTDC/MAT/114397/2009,
UTAustin/MAT/0057/2008, and UTA-CMU/MAT/0007/2009.
} \\[2mm]
{\footnotesize
Center for Mathematical Analysis, Geometry, and Dynamical Systems,\\
Departamento de Matem\'{a}tica, \\
Instituto Superior T\'{e}cnico, Avenida Rovisco Pais 1, 1049-001 Lisboa, Portugal\\

gterrone@math.ist.utl.pt}
\end{center}

{\footnotesize
\noindent
{\bf Abstract.}
We study existence and uniqueness of bounded solutions to a fractional sublinear elliptic equation with a variable coefficient, in the whole space. Existence is investigated in connection to a certain fractional linear equation, whereas the proof of uniqueness relies on uniqueness of solutions to an associated fractional porous medium equation with variable density.
\\[3pt]
\\ {\small\bf AMS subject classification:}  35A01, 35A02, 35J61, 35K65, 35R11\,.
\\ {\small\bf Keywords:} Fractional Laplacian, sublinear elliptic equations, Riesz potential, well-posedness, porous medium equation. }

%{\small\bf AMS (MOS) subject classification:} This is optional.
%But please supply them whenever possible.

\vskip.2in

%\tableofcontents

\section{Introduction}
We are concerned with existence and uniqueness of
bounded solutions to the following fractional sublinear equation:
\begin{equation}
\label{19021}
(-\Delta)^\frac{\sigma}{2}u = \rho\, u^\alpha \qquad\text{in }\Rr^N.
\end{equation}
The nonlocal operator $(- \Delta)^{\frac{\sigma}{2}}$ is the fractional Laplacian of order $\sigma/2$ with $\sigma\in (0,2), N\geq 2$. Thus the following representation in terms of a singular integral holds:
\begin{equation}
\label{ea1}
(-\Delta)^{\sigma/2} g(x)=C_{N,\sigma}\,\, \textrm{P.V.}\, \int_{\Rr^N} \frac{g(x)-g(z)}{|x-z|^{N+\sigma}}\dd z,
\end{equation}
for any $g$ belonging to the Schwartz class, where $C_{N,\sigma}$ is an appropriate positive normalization constant depending on $N$ and $\sigma$ (see \eqref{20031}). The function $\rho$ is nonnegative and bounded in $\Rr^N$, and decays at infinity fast enough; furthermore, $0<\alpha<1$.
If we replace the nonlocal operator in \eqref{19021} by the Laplace operator $\Delta$, then we obtain the following sublinear elliptic equation:
\begin{equation}
\label{18031}
-\Delta u = \rho\, u^\alpha \quad\text{in }\Rr^N\,,
\end{equation}
which, together with its counterpart in bounded domains of $\Rr^N$ completed with Dirichlet boundary conditions, has been extensively studied in the literature (see \cite{BK}, \cite{BO}, \cite{Ed}, \cite{Egn}, \cite{Kras}, \cite{Nai}, \cite{Rad}). In particular, in \cite{BK},  existence and uniqueness of bounded solutions to equation \eqref{18031} have been established, under the assumption $\rho\in L^\infty_{loc}(\Rr^N)$, $\rho\geq 0$. More precisely, in \cite{BK} it has been shown that existence of solutions to problem \eqref{18031} holds if and only if the linear problem
\begin{equation}
\label{18032}
-\Delta U = \rho \qquad\text{in }\Rr^N
\end{equation}
admits a bounded solution; moreover, the solution is unique in the class of solutions $v$ satisfying $\liminf_{|x|\to \infty} v(x)=0$\,. Whereas, asymptotic behavior as $|x|\to\infty$ of solutions to equation \eqref{18031} has been addressed in \cite{Ed}, \cite{Egn} and \cite{Nai}, under appropriate assumptions on $\rho$.

Recently, also the analysis of fractional semilinear elliptic equations have attracted the attention of various authors (see, $e.g.$,  \cite{BCdPS}, \cite{BCdP}, \cite{CT}, \cite{CS1}, \cite{CS2}, \cite{CDDS}, \cite{Tan}). In particular, for further references we point out that in \cite{BCdPS}, \cite{BCdP} existence and multiplicity of solutions  have been studied for the problem
\begin{equation*}
    \begin{cases}
   (-\Delta)^\frac{\sigma}{2}u =  u^p + \lambda  u^q  & x\in D \\
     u=0 & x\in\partial D,
    \end{cases}
\end{equation*}
where $D\subset \Rr^N$ is a bounded domain with smooth boundary
$\partial D$, $0<q\leq 1, 1<p\leq \frac{N+\sigma}{N-\sigma},
N>\sigma,\lambda>0$. To the best of our knowledge,  fractional
sublinear equations in the all $\Rr^N$, such as \eqref{19021},
have not been studied so far. Clearly, as well as in the local
case (that is, for $\sigma=2$), when considering this type of
equations in the all $\Rr^N$ several differences with respect to
the case of bounded domains occur.

The analysis of the elliptic equation \eqref{18031} is strictly related (see \cite{KRV}) to the asymptotic behavior of solutions of the Cauchy problem for the porous medium equation with variable density $\rho$:
\begin{equation}
\label{06111ii}
    \begin{cases}
   \rho\, \partial_t u = \Delta \left[ u^m\right]   & x\in \Rr^N, \quad t>0\\
     u=u_0 & x\in \Rr^N, \quad t=0
    \end{cases}
\end{equation}
with $m=1/\alpha, \rho>0$.
The question if analogous results hold for problem
\begin{equation}
\label{06111i}
    \begin{cases}
   \rho\, \partial_t u + (- \Delta)^{\frac{\sigma}{2}}\left[ u^m\right] = 0  & x\in \Rr^N, \quad t>0\\
     u=u_0 & x\in \Rr^N, \quad t=0
    \end{cases}
\end{equation}
is not the content of the present work, and still remains to be answered.

On the other hand, let us mention that in the following, we shall use existence and uniqueness results proved in \cite{PT2} for problem \eqref{06111i} (see also \cite{PT1}), in order to show uniqueness of solutions to \eqref{19021}.

\smallskip

We describe next how the paper is organized and outline the main contributions. In Section \ref{sec:back} we recall the
needed mathematical background about the fractional Laplacian, its realization through the harmonic extension, both in bounded domains and in the whole space, and give the precise notion of solution we will deal with. As well as in \cite{BCdPS}, \cite{BCdP} we will consider energy solutions. Consequently, we need to suppose that $N>2\sigma$ to prove some results (see Remark \ref{vw}). In Section \ref{sec:linp} we perform a detailed and self-contained analysis of the linear problem
\begin{equation}
\label{19023i}
(-\Delta)^\frac{\sigma}{2}U = \rho \qquad\text{in }\Rr^N,
\end{equation}
establishing existence and uniqueness of solutions. Boundedness of solutions and behavior at infinity of solutions is investigated as well; in particular a decay estimate at infinity
is shown using some results in \cite{Rubin}.
In Section \ref{sec:exi} we study existence of solutions to equation \eqref{19021}. More precisely, we prove that existence of bounded solutions to the linear equation
\eqref{19023i} is sufficient (see Theorem \ref{190217})
to existence of solutions to \eqref{19021}. This somehow rephrases, in the nonlocal framework, some results obtained in \cite{BK} for the local problem \eqref{18031}. Finally in Section \ref{sec:uni}, by exploiting uniqueness results for problem \eqref{06111i} proved in \cite{PT2}, we show uniqueness of solutions of  \eqref{18031} satisfying a decay condition at infinity; see Theorem \ref{02034}.

%%% SECTION:  %%%
\section{Mathematical background}
\label{sec:back}
We always make the following assumption:
\begin{equation}
\label{A0} \tag{{\bf A}$_0$}
\begin{cases}
\text{(i)} &\rho\in L^\infty(\Rr^N), \,\rho\geq0 \text{ a.e. in } \Rr^N, \, \rho \not\equiv 0\\
\text{(ii)} &0<\sigma <2\\
\text{(iii)} & 0<\alpha<1.
\end{cases}
\end{equation}
Furthermore, about $\rho=\rho(x)$, we suppose that the following decay condition at infinity holds:
\begin{equation}
\label{A2b} \tag{{\bf A}$_1$}
\begin{array}{c}
  \textrm{there exist}\; \check C>0, \check R>0 \; \textrm{and}\; \beta\in (N, \infty) \; \textrm{such that}\\
  \rho(x)\leq \check C  |x|^{- \beta}\quad \textrm{for all}\;\, x\in \Rr^N\backslash B_{\check R}.
\end{array}
\end{equation}

Let us introduce the following sets:
\begin{gather*}
L^1_\rho(\Rr^N):= \left\{ f \text{ measurable in }\Rr^N\, \Big{|}\, \|f\|_{L^1_{\rho}(\Rr^N)}:=\int_{\Rr^N} f \, \rho\, \dd x <\infty\right\},\\
L^+_\rho(\Rr^N):= \left\{ f\in L^1_\rho(\Rr^N)\, \Big{|}\, f\geq 0\right\}.\\
\end{gather*}
The fractional Laplace operator $(-\Delta)^{\sigma/2}$ can be defined in many different ways, one of which relies on the Fourier transform. For any $g$ in the class of Schwartz functions, if $(-\Delta)^{\sigma/2} g = h$ then
\begin{equation}
\label{02015}
\hat{h}(\xi)= |\xi|^{\sigma}\hat{g}(\xi).
\end{equation}
When, as in the present situation, $\sigma$ varies in the open interval $(0,2)$ then representation \eqref{ea1} holds.

%Furthermore, if $\varphi$ is a smooth and
%bounded function defined in $\Rr^N$, we can consider its $\sigma$--harmonic
%extension $v=\D{E}(\varphi)$ to the upper half-space
%\[
%\Omega:= \Rr^{N+1}_+ = \{(x,y): \, x\in \Rr^N, \, y>0 \},
%\]
%that is, the unique smooth and bounded solution $v(x,y)$ of the problem
%\[
%  \begin{cases}
%    \D{div} \big\{y^{1-\sigma}\nabla v\big\}=0 &\text{in }\Omega\\
%    v(x,0)=\varphi(x) & \text{in }\Gamma.
%    \end{cases}
%\]
%Here $\Gamma:=\overline \Omega \cap \{y=0\}\equiv \Rr^N$. It has been proved (see %\cite{CS1},
%\cite{CafS},  \cite{V2}) that
%\[
%-\frac{\partial v}{\partial y^\sigma}(x,0)=\left(-\Delta \right)^{\frac \sigma 2} \varphi(x)\quad
%\textrm{for all } x\in \Gamma,
%\]
%where $$\mu_\sigma:=\frac{2^{\sigma-1}\Gamma(\sigma/2)}{\Gamma(1-\sigma/2)}, \;  \frac{\partial v}{\partial y^\sigma}(x,0):=\mu_\sigma \lim_{y\to 0^+} y^{1-\sigma}\frac{\partial v}{\partial y}. $$

\subsection{Problem in the all space}\label{pas}
Multiplying the nonlocal partial differential equation in \eqref{19021} by a test function $\varphi$ compactly supported in $\Rr^N$, integrating by parts, taking into account \eqref{02015} and using the Plancherel's Theorem, we discover that
\begin{equation}
\label{02016ter}
\int_{\Rr^N} \rho\, u^\alpha \varphi\, \dd x - \int_{\Rr^N}(-\Delta)^{\sigma/4}u\,
(-\Delta)^{\sigma/4}\varphi \, \dd x = 0.
\end{equation}

We denote by $\dot{H}^{\sigma/2}(\Rr^N)$ the fractional Sobolev space obtained by completing $C^\infty_0(\Rr^N)$ with the norm
$\|\psi\|_{\dot{H}^{\sigma/2}}= \|(-\Delta)^{\sigma/4} \psi\|_{L^2(\Rr^N)}$.

\begin{definition}\label{defsol01}
A \emph{solution} to equation \eqref{19021} is a function $u\geq 0$ such that:
\begin{itemize}
\item{} $u\in \dot{H}^{\sigma/2}(\Rr^N)\cap L^{\infty}(\Rr^N)\,$;
\item for any $\varphi\in C^{\infty}_0(\Rr^N)$ identity \eqref{02016ter} holds. \end{itemize}
\end{definition}

\subsection{Problem in bounded domains}
\label{subs_pbd}
Let $D$ be a bounded domain in $\Rr^N$ with smooth boundary $\partial D$. We use next a spectral decomposition to define the fractional operator $(-\Delta)^{\sigma/2}$ in $D$. Let $\{\xi_n\}_1^\infty$  be an orthonormal basis of $L^2(D)$  made by eigenfunctions of $-\Delta$ in $D$ completed with homogeneous Dirichlet boundary conditions, and let $\{\lambda_n\}_1^\infty$
be the sequence of the corresponding eigenvalues. For any $u\in C^{\infty}_0(D)$
\[ (-\Delta)^{\sigma/2}u := \sum_{n=1}^\infty \lambda_n^{\sigma/2} \, u_n \, \xi_n\quad \textrm{in } D\,, \]
where $u=\sum_{n=1}^\infty u_n\,  \xi_n$ in $L^2(D).$
By density, $(-\Delta)^{\sigma/2} u$ can be also defined for $u$ belonging to the Hilbert space
\[
H_0^{\sigma/2}(D):= \left\{u\in L^2(D)\,\,\Big{|}\,\, \|u\|^2_{H_0^{\sigma/2}(D)}:=\sum_{n=1}^\infty \lambda_n^{\sigma/2}u_n^2 <\infty \right\}.
\]

We then consider the problem
\begin{equation}
\label{19028}
\begin{cases}
(-\Delta)^\frac{\sigma}{2}u = \rho \, u^\alpha &\text{in } D\\
u=0 & \text{in }\partial D.
\end{cases}
\end{equation}

\begin{definition}\label{defsol02}
A \emph{solution} to problem \eqref{19028} is a function $u\geq 0$ such that:
\begin{itemize}
\item{} $u\in H_0^{\sigma/2}(D)\cap L^\infty(D)$;
\item for any $\varphi\in C^{\infty}_0(D)$ identity \eqref{02016ter} holds with $\Rr^N$ replaced by $D$.
\end{itemize}
\end{definition}

\section{The linear problem}
\label{sec:linp}
In this Section we study the linear problem
\begin{equation}
\label{19023}
(-\Delta)^\frac{\sigma}{2}U = \rho \qquad\text{in }\Rr^N.
\end{equation}

\begin{definition}\label{defsol04}
A \emph{solution} to problem \eqref{19023} is a function $U$
such that:
\begin{itemize}
\item{} $U\in \dot{H}^{\sigma/2}(\Rr^N)\cap L^\infty(\Rr^N)$;
\item for any $\varphi\in C^{\infty}_0(\Rr^N)$ there holds
\begin{equation}
\label{02016bis}
\int_{\Rr^N} \rho\, \varphi\, \dd x-  \int_{\Rr^N}(-\Delta)^{\sigma/4}U\,
(-\Delta)^{\sigma/4}\varphi \, \dd x = 0.
\end{equation}
\end{itemize}
\end{definition}

%As well as for the nonlinear problem introduced in the previous Section, we consider problem
%\begin{equation}
%\label{26022}
%    \begin{cases}
%    \D{div}\big\{y^{1-\sigma} \nabla W\big\} = 0 & (x,y)\in \Omega,\\
%    \displaystyle \frac{\partial W}{\partial y^\sigma}= \rho & x\in \Gamma,\\
%     \end{cases}
%\end{equation}
%for the harmonic extension $W=\D{E}(U)$, and give the next
%
%\begin{definition}
%\label{26025}
%A \emph{solution} to problem \eqref{26022} is a
%pair of functions $(U,W)$
%such that
%\begin{itemize}
%\item{} $ W\in X^\sigma( \Omega)\cap L^\infty(\Omega)$;
%\item{} $W|_{\Gamma}=U,\, U\in L^1_\rho(\Rr^N)$;
%\item{} for any $\varphi\in C^\infty_0(\bar\Omega)$ there holds
% \begin{equation}
%    \label{26024}
%    \int_{ \Gamma} \rho\, \varphi(x,0) \,\dd x =  {\mu_\sigma}   \int_{\Omega} y^{1-\sigma}\langle\nabla\varphi,  \nabla W\rangle \,\dd x\, \dd y.
%    \end{equation}
%\end{itemize}
%\end{definition}

We also introduce the linear problem in a bounded domain $D \subset\Rr^N$,
\begin{equation}
\label{26026}
\begin{cases}
(-\Delta)^\frac{\sigma}{2}U = \rho &\text{in } D\\
U=0 & \text{in }\partial D.
\end{cases}
\end{equation}

\begin{definition}\label{defsol06}
A \emph{solution} to problem \eqref{26026} is a function $U$ such that:
\begin{itemize}
\item{} $u\in {H}_0^{\sigma/2}(D)\cap L^{\infty}(D)$;
\item for any $\varphi\in C^{\infty}_0(D)$ identity \eqref{02016bis} holds with $\Rr^N$ replaced by $D$.
\end{itemize}
\end{definition}

%The associated extension problem in the half-cylinder $\mathcal{C}_D:= D\times (0, \infty)$ with zero lateral condition is:
%\begin{equation}
%\label{26027}
%    \begin{cases}
%    \D{div}\big\{y^{1-\sigma} \nabla W\big\} = 0 & (x,y)\in \mathcal{C}_D,\\
%    W=0 & (x,y)\in \partial_L\mathcal{C}_D,\\
%    \displaystyle \frac{\partial W}{\partial y^\sigma}=\rho & x\in D.
%     \end{cases}
%\end{equation}
%
%\begin{definition}
%\label{26028}
%A \emph{solution} to problem \eqref{26027} is a
%pair of functions $(U,W)$ such that
%\begin{itemize}
%\item{} $ W\in X^\sigma_0(\mathcal{C}_D)$;
%\item{} $W|_{D}=U$;
%\item{} for any $\varphi\in C^\infty_0(\overline{\mathcal{C}_D}), \varphi=0$ on $\partial_L \mathcal C_D$, there holds
% \begin{equation}
%    \label{26029}
%    \int_{ D} \rho\, \varphi(x,0) \,\dd x =  {\mu_\sigma}   \int_{\mathcal{C}_D} y^{1-\sigma}\langle\nabla\varphi,  \nabla W\rangle \,\dd x\, \dd y.
%    \end{equation}
%\end{itemize}
%\end{definition}
%

We introduce the following property:
\begin{equation}
\label{H1}
\tag{{\bf {H}}}
%\begin{array}{c}
%  \textrm{$\rho$ is such that}\\
  \textrm{there exits a solution $U$ to
\eqref{19023}.}
%\\  \textrm{satisfying $U^\alpha\in L^1_f(\Rr^N)$}.
%\end{array}
\end{equation}

%
%\begin{remark}\label{ossd2}
%As described in Section \ref{pas}, condition \eqref{H1} is equivalent to the existence of a solution to problem \eqref{19023}, in the sense of Definition \ref{defsol04}.
%\end{remark}

\subsection{Existence of solutions to the linear problem}
Let
\[
K^\sigma(x):= C_{N,\sigma}\frac 1{|x|^{N-\sigma}}\quad (x\in \Rr^N\setminus\{0\})
\]
be the Riesz kernel, where
\begin{equation}\label{20031}
C_{N,\sigma}:=2^{\sigma-1}\sigma\frac{\Gamma\big((N+\sigma)/2\big)}{\pi^{N/2}\Gamma\big((1-\sigma)/2\big)}\,.
\end{equation}
As well as for the standard Laplace operator, a solution to problem \eqref{19023} can be constructed by convolving such a kernel with the function $\rho$, that is
\[
(K^\sigma * \rho) (x) = C_{N,\sigma} \int_{\Rr^N} \frac{\rho(y)}{|x-y|^{N-\sigma}}\dd y\,.
\]

We are interested in determining conditions to be imposed on $\rho$ such that property \eqref{H1} is satisfied.

\begin{remark}
\label{remST}
\begin{itemize}
\item[(i)] It is direct to check that  (see, $e.g.$, \cite{ST})  $\big(K^\sigma * \rho\big)(x)$ is finite at any $x\in \Rr^N$ if and only if
\[
\int_{\Rr^N} \frac{|\rho(y)|}{1+|y|^{N-\sigma}}\dd y <\infty \quad \text{ and }\quad
\int_{B(x, 1)} \frac{|\rho(y)|}{|x-y|^{N-\sigma}}\dd y \text{ for any }x\in \Rr^N.
\]

\item[(ii)] As a consequence of (i), when $\rho\in L^{\infty}_{loc}(\Rr^N)$,
$\big(K^\sigma * \rho\big)(x)$ is finite at any $x\in \Rr^N$ if and only if it is finite at some $x_0\in \Rr^N\,.$

\item[(iii)]  From (i) it immediately follows that if $\rho\in L^{\infty}(\Rr^N)$ and $\big(K^\sigma * \rho\big)(0)<\infty$, then
$K^\sigma * \rho\in L^\infty(\Rr^N)$. In fact, in this case
\[
\int_{B(x, 1)} \frac{|\rho(y)|}{|x-y|^{N-\sigma}}\dd y \leq \|\rho\|_{L^\infty(\Rr^N)} \int_{B(0,1)} \frac{\dd z}{|z|^{N-\sigma}},
\]
for any $x\in \Rr^N$.
\item[(iv)] From $(iii)$ it follows that when \eqref{A0} and
\eqref{A2b} are satisfied, then $K^\sigma *\rho\in
L^\infty(\Rr^N)$.
\end{itemize}
\end{remark}
It is direct to obtain the next
\begin{proposition}
\label{27021} Assume \eqref{A0}--(i),(ii). If
\begin{equation}
\label{eea1} K^\sigma * \rho\in L^\infty(\Rr^N)\cap \dot
H^{\sigma/2}(\Rr^N)\,,
\end{equation}
then $K^\sigma * \rho$ solves equation \eqref{19023}. Thus,  property \eqref{H1} is satisfied.
\end{proposition}

Moreover, we have the following result.
\begin{proposition}
\label{oss1b} If \eqref{A0}-\eqref{A2b} are verified and
$N>2\sigma$, then condition \eqref{eea1} holds.
\end{proposition}

\noindent{\it Proof\,.\,\, } For any $p>1$, by the
Hardy-Littlewood-Sobolev inequality (see, $e.g.$, \cite{AH}),
\begin{equation}\label{er1}
\|K^\sigma* \rho\|_{L^{p^*}(\Rr^N)}\leq C_p
\|\rho\|_{L^p(\Rr^N)}\,,
\end{equation} where $p^*=\frac{Np}{N-\sigma p}$ and $C_p$ is a proper positive
constant. In view of $\eqref{A0}-\eqref{A2b}$, $\rho\in
L^p(\Rr^N)$ for any $p\in [1,\infty]$. This combined with
\eqref{er1} and the hypothesis $N>2\sigma$ implies that both
$K^\sigma* \rho$ and $\rho$ belong to $L^2(\Rr^N)$. Thus, from
Proposition 3.1.7 and Theorem 1.1.1 of \cite{AH} it is immediate
to deduce that $K^\sigma* \rho\in \dot H^{\sigma/2}(\Rr^N)$.
Hence, from Proposition \ref{27021} and Remark \ref{remST}-$(iv)$
the conclusion follows. \hfill $\square$

\medskip

From \cite{Rubin} the following lemma can be deduced (see \cite[Corollary 5.4]{PT2}).

\begin{lemma}\label{lemmaRub2}
Let $N\geq 2.$ Let assumptions \eqref{A0}, \eqref{A2b} be satisfied.  Then
\begin{equation}\label{ea58}
(K^\sigma * \rho) (x)  \,\to \,0\quad
\textrm{as}\;\; |x|\to \infty\,.
\end{equation}
More precisely, for some $C>0$, we have:
\begin{equation}\label{ea60}
(K^\sigma * \rho) (x) \leq C
|x|^{\sigma-\nu-\frac N r}\quad \textrm{for all}\;\; x\in \Rr^N,
\end{equation}
provided $\frac N2 (2-\sigma)<\nu<N$ and \begin{equation}\label{ea59}
\max\left\{\frac 2{\sigma},\frac N{\beta-\nu}\right\}<r<\frac
N{N-\nu}\,.
\end{equation}
\end{lemma}

%\begin{lemma}
%\label{260212}
%Suppose that there exists $\sigma <\beta < N$ ( $< N$???) such that
%\[
%\rho(x)\leq C |x|^{-\beta} \quad\text{for any }|x|>1.
%\]
%Then
%\[
%(K^\sigma * \rho)^\alpha\in L^1_\rho(\Rr^N).
%\]
%provided
%\begin{equation}
%\label{250213}
%\beta > \max \left\{ \frac{N+\sigma\alpha}{\alpha+1}, \, \frac{2N - N\alpha\sigma + 2\sigma\alpha}{2}\right\}.
%\end{equation}
%\end{lemma}
%Notice that
%the maximum in \eqref{250213} is smaller than $N$, then the set of admissible values for $\beta $ is nonempty.
%\begin{proof}
%Let us write
%\begin{multline*}
%\int_{\Rr^N}\!\!\int_{\Rr^N}\left(\frac{\rho(y)}{|x-y|^{N-\sigma}}\right)^\alpha \rho(x)\, \dd x\, \dd y =
%\int_{B(0,1)}\!\int_{\Rr^N}\left(\frac{\rho(y)}{|x-y|^{N-\sigma}}\right)^\alpha \rho(x)\, \dd x\, \dd y \\
%+\int_{\Rr^N\backslash B(0,1)}\!\int_{\Rr^N}\left(\frac{\rho(y)}{|x-y|^{N-\sigma}}\right)^\alpha \rho(x)\, \dd x\, \dd y.
%\end{multline*}
%The first integral in the right hand side is bounded, whereas by hypothesis and Remark \ref{oss2},
%\begin{multline*}
%\int_{\Rr^N\backslash B(0,1)}\!\int_{\Rr^N}\left(\frac{\rho(y)}{|x-y|^{N-\sigma}}\right)^\alpha \rho(x)\, \dd x\, \dd y
%\leq  C \int_{\Rr^N\backslash B(0,1)} \frac{1}{|x|^{(\frac{N}{r}-\sigma)\alpha + \beta}}.
%\end{multline*}
%which is bounded if and only if $r<\frac{N\alpha}{N-\beta + \sigma \alpha}$. Then, the statement follows taking into account condition \eqref{260211}.
%\end{proof}

For further references, let us consider problem
\begin{equation}
\label{260214}
\begin{cases}
(-\Delta)^\frac{\sigma}{2}U_R = \rho &\text{in } B_R\\
U_R=0 & \text{in }\partial B_R.
\end{cases}
\end{equation}
Suppose that \eqref{A0} is satisfied. Note that for each $R>0$
problem \eqref{260214} admits a unique solution $U_R$; moreover,
by strong maximum principle (see \cite{CDDS}), $U_R> 0$ in $B_R$,
and
\[
U_R(x)= \int_{B_R} G_R(x, y) \rho(y)\, \dd y
\]
where $G_R$ is the Green function of the operator $(-\Delta)^{\sigma/2}$ in the domain $B_R$, completed with zero Dirichlet boundary conditions on $\partial B_R$ (see, $e.g.$ \cite{dBV}).

Since $U_R\geq 0$ in $B_R$ for any $R>0$, if $R_1<R_2$, then $U_{R_2}$ is a supersolution to problem
\[
\begin{cases}
(-\Delta)^\frac{\sigma}{2} U = \rho &\text{in } B_{R_1}\\
U =0 & \text{in }\partial B_{R_1}.
\end{cases}
\]
So, by comparison principles,
\[ 0\leq U_{R_1}\leq U_{R_2}\quad \textrm{in}\;\; B_{R_1}\,. \]
By results in \cite{dBV}, $G_R(x, y)\leq \tilde C
K^\sigma(|x-y|)$, for some positive constant $\tilde C$
independent of $R$. So
\begin{equation}
\label{ed2} U_R(x) \leq \tilde C(K^\sigma * \rho) (x)\quad (x\in
\Rr^N)\,.
\end{equation}

\subsection{Uniqueness for the linear problem}
\begin{lemma}
\label{27026}
Let assumptions \eqref{A0}, \eqref{A2b} be satisfied; let $N>2\sigma$. If $U$  is a bounded solution to problem \eqref{19023}, such that $U(x)\to 0$ as $|x|\to \infty$, then it coincides with $K^\sigma * \rho$.
\end{lemma}

\begin{proof}
Set
\[
\Omega:= \Rr^{N+1}_+ = \{(x,y): \, x\in \Rr^N, \, y>0 \}, \qquad \Gamma:=\overline \Omega \cap \{y=0\}\equiv \Rr^N,
\]
and let $X^\sigma(\Omega)$ be the completion of $C^\infty_0(\bar\Omega)$ with the norm
\[
\|v\|_{X^\sigma(\Omega)} = \left(\mu_\sigma \int_{\Omega} y^{1-\sigma}|\nabla v|^2
\, \dd x \, \dd y \right)^\frac{1}{2},
\]
where $\mu_\sigma:=\frac{2^{\sigma-1}\Gamma(\sigma/2)}{\Gamma(1-\sigma/2)}$. Given a function $f\in X^\sigma(\Omega)$ we denote by $f |_{\Gamma}$ its trace on $\Gamma$.

Let  $W:=\D{E}(U)$ be the $\sigma$--harmonic extension of $U$ to the upper half-space $\Omega$
that is, the unique smooth and bounded solution $W(x,y)$ of the problem
\begin{equation}
\label{26022}
    \begin{cases}
    \D{div}\big\{y^{1-\sigma} \nabla W\big\} = 0 & (x,y)\in \Omega,\\
    \displaystyle \frac{\partial W}{\partial y^\sigma}= \rho & x\in \Gamma,
     \end{cases}
\end{equation}
where
\[
\frac{\partial W}{\partial y^\sigma}(x,0):=\mu_\sigma \lim_{y\to 0^+} y^{1-\sigma}\frac{\partial W}{\partial y}.
\]
A solution to problem \eqref{26022} is a
pair of functions $(U,W)$ such that $W\in X^\sigma( \Omega)\cap L^\infty(\Omega)$, $W|_{\Gamma}=U$ and for any $\varphi\in C^\infty_0(\bar\Omega)$ there holds
\begin{equation}
    \label{26024}
    \int_{ \Gamma} \rho\, \varphi(x,0) \,\dd x =  {\mu_\sigma}   \int_{\Omega} y^{1-\sigma}\langle\nabla\varphi,  \nabla W\rangle \,\dd x\, \dd y.
\end{equation}

Set also $\bar W:= \D{E}(K^\sigma * \rho)$ and $\tilde{W}:= W- \bar W$.
Take a sequence ${\varphi_n}\subset C^\infty_0\big(\bar \Omega\big)$ such that $\varphi_n\to \tilde W$ as $n\to \infty$ in $X^\sigma_0(\Omega)$.
By \eqref{26024}, for all $n\in \mathbb N$, we have
\[
\int_\Omega y^{1-\sigma}\langle\nabla \tilde{W}, \nabla \varphi_n \rangle \dd x\, \dd y = 0.
\]
Sending $n\to\infty$, we get
\[
\int_\Omega y^{1-\sigma}|\nabla \tilde{W}|^2 \dd x\, \dd y = 0,
\]
so $\tilde W$ is constant in $\Omega$. Furthermore,  in view of \eqref{ea58}, we have that $\bar W(x,0)=(K^\sigma * \rho)(x)\to 0$ as $|x|\to \infty$. Thus, taking into account that by assumption $W(x,0)=U(x)\to 0$ as $|x|\to\infty$, we deduce that $\tilde{W}(x,0)\to 0$ as $|x|\to \infty$. This implies the identity $\tilde{W}\equiv 0$ and thus the statement.
\end{proof}

\begin{lemma}
\label{lemma*} Let assumptions \eqref{A0}, \eqref{A2b} be satisfied; let $N>2\sigma.$
Let $U$ be a solution to
\begin{equation}
\label{27028}
(-\Delta)^{\sigma/2} U \leq \rho \qquad\text{in }\Rr^N
\end{equation}
such that $U(x)\to 0$ as $|x|\to\infty$.
In addition, suppose that $f:=(-\Delta)^{\sigma/2} U\in L^\infty(\Rr^N)$.
Then
\begin{equation}
\label{27029} U\leq K^\sigma*\rho \qquad \text{in }\Rr^N.
\end{equation}
\end{lemma}

\begin{proof}
Let $g:= (-\Delta)^{\sigma/2}  (K^\sigma*\rho- U)$. Thus
\begin{equation}\label{e7b}
g\geq 0 \qquad\text{in }  \Rr^N.
\end{equation}
Consider the equation
\begin{equation}\label{e6b}
(-\Delta)^{\sigma/2} V = g \qquad\text{in }\Rr^N\,.
\end{equation}
Note that $g=\rho-f.$ Thus, in view of hypothesis \eqref{A2b} and
the fact that $f\in L^\infty(\Rr^N)$, from Remark \ref{remST} we
can infer that $K^\sigma * g\in L^\infty(\Rr^N)$. Furthermore,
since $0\leq g \leq \rho$, and $\rho\in L^1(\Rr^N)$, it also
follows that $K^\sigma * g\in L^1_g(\Rr^N)$. So, by Proposition
\ref{27021}, $K^\sigma * g$ is a bounded solution to equation
\eqref{e6b}. Since $g\leq \rho$, from Lemma \ref{lemmaRub2} and
\eqref{e7b} it follows that $\big(K^\sigma * g\big)(x)\to 0$ as
$|x|\to\infty$. Clearly, $K^\sigma*\rho-U$ is a bounded solution
to equation \eqref{e6b} such that
$\Big[\big(K^\sigma*\rho\big)(x)-U(x)\Big]\to 0$ as
$|x|\to\infty$. From Lemma \ref{27026} we deduce that
\[ K^\sigma*\rho - U = K^\sigma * g \quad \textrm{in}\;\; \Rr^N\,. \]
Now, from \eqref{e7b} the conclusion follows.
\end{proof}

\begin{remark}\label{vw}
Note that since the proof of Lemma \ref{lemma*} uses $K^\sigma * \rho\in \dot H^{\sigma/2}(\Rr^N)$, we need the hypothesis $N>2\sigma$. 
We stress the fact that the content of such lemma will be important to prove Theorem \ref{tuniell}, which deals with $\rho\geq 0$. 
\end{remark}

%%% SECTION:  %%%
\section{Existence results}
\label{sec:exi}
Our goal is to prove the following:
\begin{theorem}
\label{190217}
Let assumption \eqref{A0} be satisfied. If %$K^\sigma * \rho\in L^\infty(\Rr^N)\cap L^1_{\rho}(\Rr^N)$,
\eqref{eea1} holds and $K^\sigma * \rho\in L^1_{\rho}(\Rr^N)$,
then there exists a solution $u$ to problem \eqref{19021}.
Furthermore, for some $C>0$,
\begin{equation}\label{er10}
u(x)\leq C |x|^{\sigma-\nu-\frac N r}\quad \textrm{for all}\;\;
x\in \Rr^N\,,
\end{equation}
with $\nu, r$ as in Lemma \ref{lemmaRub2}.
\end{theorem}

\begin{remark}\label{ossd3}
In view of Proposition \ref{oss1b}, we have that if $\eqref{A0}$
and \eqref{A2b} are satisfied, and $N>2\sigma$, then \eqref{eea1} holds and
$K^\sigma * \rho\in L^1_{\rho}(\Rr^N)$.
%If $K^\sigma * \rho\in L^\infty(\Rr^N)\cap L^1_{\rho}(\Rr^N)$, then there exists a solution to problem \eqref{19021}.
\end{remark}

In the sequel, for any $R>0$,  we shall make use of problem
%\begin{equation}
%\label{190214}
%    \begin{cases}
%    \D{div}\big\{y^{1-\sigma} \nabla w_R\big\} = 0 & (x,y)\in \mathcal{C}_R:= B_R\times (0, \infty),\\
%    w_R=0 & x\in \partial\mathcal{C}_R,\, y>0,\\
%    \displaystyle \frac{\partial w_R}{\partial y^\sigma}= \rho u_R^\alpha & x\in B_R.
%     \end{cases}
%\end{equation}

\begin{equation}
\label{190214}
\begin{cases}
(-\Delta)^\frac{\sigma}{2}u_R = \rho\, u_R^\alpha&\text{in }B_R\\
u_R=0 & \text{in }\partial B_R.
\end{cases}
\end{equation}

Now, we state some results concerning problem \eqref{190214}, that
can be proved by standard methods (see \cite{BCdP}). To begin
with, by the classical procedure of sub-- and super solutions,
next Lemma can be deduced (see \cite[Lemma 3.1]{BCdPS} or
\cite[Lemma 4.2]{BCdP}).

\begin{lemma}
\label{190216} Let assumption \eqref{A0} be satisfied. Let $u_1$
and $u_2$ be respectively a subsolution and a supersolution to
problem \eqref{19028}, and assume that $u_1\leq u_2$ in $D$. Then
there exists $u$ solution to problem   \eqref{19028} such that
$u_1\leq u \leq u_2$ in $D$.
\end{lemma}

The following comparison result can be easily deduced by the same
arguments as in \cite[Lemma 4.3]{BCdP}\,.
\begin{lemma}
\label{190215} Let assumption \eqref{A0} be satisfied. Let $u_1$
and $u_2$ be respectively a subsolution and a supersolution to
problem \eqref{19028}, and assume that $u_1, u_2>0$. Then $u_1\leq
u_2$ in $D$.
\end{lemma}

Moreover, the next existence result holds.
\begin{proposition}
\label{190213}
Let assumption \eqref{A0} be satisfied. Then for any $R>0$ there exists a solution $u_R$ to problem
\eqref{190214}.
\end{proposition}

\begin{proof}
Consider the functional $J:  \dot H^{\sigma/2}(B_R)\to \Rr$ defined as
\[
J(w):= \frac{1}{2} \int_{B_R} \left|(-\Delta)^\frac{\sigma}{4}w\right|^2\, \dd x - \frac{1}{\alpha+1}\int_{B_R} \rho\, w^{\alpha+1}\, \dd x
\]
In view of Definition \ref{defsol02},
it is well defined, bounded from below and  coercive in $\dot H^{\sigma/2}(B_R)$. Then by standard tools, the conclusion follows.
\end{proof}

\smallskip
Now we can prove Theorem \ref{190217}.

\begin{proof}
[Proof of Theorem \ref{190217}]
For any $R>0$, by Proposition \ref{190213} a solution $u_R$ to problem \eqref{190214} exists. Moreover, by Lemma \ref{190215} it is unique.
%Since $K^\sigma * \rho\in L^\infty(\Rr^N)\cap L^1_{\rho}(\Rr^N)$, by Proposition \ref{27021} there exists $(U, W)$ solution to problem \eqref{26022}.
By strong maximum principle,
\begin{equation}\label{17042}
%w_R>0\quad \textrm{in}\;\; \mathcal C_R,\;\,
u_R>0\quad \textrm{in}\;\; B_R\,.
\end{equation}
Observe that
\begin{equation}
\label{190218}
R<R' \quad\Rightarrow\quad u_R \leq u_{R'} \text{ in }B_R.
\end{equation}
In fact, in view of \eqref{17042}, $u_{R'}$ is a supersolution to \eqref{190214}. Then $u_R \leq u_{R'}$.

Let $U_R$ be the solution to problem \eqref{260214}. Due to
\eqref{ed2}, for  $C\geq (\tilde C\|K^\sigma * \rho
\|_{L^\infty(B_R)})^\frac{\alpha}{1-\alpha}$, $CU_R$ is a
supersolution to problem \eqref{190214}. In fact, for any
$\varphi\in C^\infty_0(B_R)$,
\[
\int_{\Rr^N}(-\Delta)^{\sigma/4}(CU_R)\,
(-\Delta)^{\sigma/4}\varphi \, \dd x = \int_{ \Rr^N} C\, \rho\,
\varphi \,\dd x \geq  \int_{ \Rr^N} \rho\, (CU_R)^\alpha\,
\varphi \,\dd x.
\]

Hence,
\begin{equation}
\label{190219} 0\leq u_R \leq CU_R \quad\text{in }B_R.
\end{equation}

From \eqref{190218}, \eqref{190219} and \eqref{ed2} it follows
that there exists
\begin{gather*}
%w:= \lim_{R\to \infty} w_R \quad\text{in }\Omega,\\
u:=  \lim_{R\to \infty} u_R \quad\text{in }\Rr^N;
\end{gather*}
furthermore, $u\in L^\infty(\Rr^N)$ and \eqref{er10} holds true.

\smallskip

For each $R>0$, take a sequence $\{\varphi_n\}\subset C^\infty_0(B_R), \varphi_n\to u_R$ in $\dot H^{\sigma/2}_0(B_R)$ as $n\to\infty$.
From Definition  \ref{defsol02}, for each $R>0$, for all $n\in \mathbb N$, we have:
\[
\int_{B_R} (-\Delta)^{\sigma/4}u_R\,
(-\Delta)^{\sigma/4}\varphi_n \, \dd x =\int_{B_R} \rho\,  u^{\alpha}_R\,  \varphi_n\,  \dd x.
\]

Letting $n\to\infty$ we obtain:
\begin{equation}\label{19031}
\int_{B_R} |(-\Delta)^{\sigma/4}u_R|^2 \, \dd x = \int_{B_R} \rho \, u_R^{\alpha+1}\,  \dd x\,.
\end{equation}
Take any open subset $V\subset\Rr^N$ and select $R_0>0$ so big
that $V \subset B_{R_0}$. From \eqref{190219}, \eqref{ed2} and
\eqref{H1}, since $K^\sigma*\rho\in L^\infty(\Rr^N)\cap
L^1_{\rho}(\Rr^N) \cap \dot H^{\sigma/2}(\Rr^N)$, there exists
$C>0$, independent of $R$, such that
\begin{equation}
\label{190222}
\int_{B_R} \rho\, u_R^{\alpha+1}\, \dd x \leq C \qquad\text{for any }R>0.
\end{equation}
Therefore \eqref{19031} implies
\begin{equation}
\label{190223}
\int_{B_R} |(-\Delta)^{\sigma/4}u_R|^2 \, \dd x  \leq C \qquad\text{for any }R>0.
\end{equation}
By letting $R\to \infty$ in \eqref{190222} and \eqref{190223}, we
obtain \eqref{02016ter}, $u\in \dot H^{\sigma/2}(\Rr^N)\cap
L^\infty(\Rr^N)$. This completes the proof.
\end{proof}

%\begin{remark}
%\label{solmin} In the proof of Theorem \ref{190217} we have constructed a solution $(u,w)$. Such solution turns out to be  {\it minimal}, in the sense that if  $(\tilde u,\tilde w)$  is another solution, then $u\leq \tilde u$ and $w\leq \tilde w$. Moreover, by construction it follows that $u>0$ in $\Gamma$, $w>0$ in $\Omega$\,.
%\end{remark}

\begin{remark}
\label{rem1} Assume that $\rho$ satisfies \eqref{A0}, \eqref{A2b}; let $N>2\sigma$.
Then the solution constructed in Theorem \ref{190217} satisfies
the following identity:
\[
u(x)=\int_{\Rr^N} \frac{\rho(y) \, u^\alpha (y)}{|x-y|^{N-\sigma}}\dd y.
\]
In fact, let $f(x):= \rho(x) u^\alpha(x) \; (x\in \Rr^N)$, and consider the equation
\begin{equation}\label{19032}
(-\Delta)^{\sigma/2} v = f \quad\text{in }\Rr^N.
\end{equation}
Since $u\in L^\infty(\Rr^N)$ and $\rho$ satisfies \eqref{A0},
\eqref{A2b}, we have that $v=K^\sigma * f\in L^\infty(\Rr^N)\cap
L^1_f(\Rr^N)$. By Proposition \ref{27021}, $v$ is a solution to
equation \eqref{19032}. Since $u$ is nonnegative and bounded, from
Lemma \ref{lemmaRub2} it follows that $v(x)\to 0$ as
$|x|\to\infty$. Clearly, the same holds for $u$ (see
\eqref{er10}). Thus, from Lemma \ref{27026} the conclusion
follows.
 \end{remark}

\begin{remark}
\label{rem2}
The dependence of the solution of problem \eqref{19021} upon $\rho$ is monotone increasing. In fact, if $\rho_1\leq \rho_2$ and $u_1$ and $u_2$ are the corresponding solutions of \eqref{19021}, then $u_2$ is a supersolution to
\[
\begin{cases}
(-\Delta)^\frac{\sigma}{2}u = \rho_1\, u^\alpha &\text{in }B_R\\
u=0 &\text{on }\partial B_R.
\end{cases}
\]
Thus the sequence $v_{1, R}$ of function approximating $u_1$ satisfy $v_{1,R}\leq u_2$ in $B_R$. Passing to the limit as $R\to \infty$ we get $u_1\leq u_2$.
\end{remark}

\section{Uniqueness results}
\label{sec:uni}
\subsection{Fractional porous medium equation with variable density}
For later use we introduce next a fractional porous medium equation and recall some results established in \cite{PT2}.
Consider the following nonlinear nonlocal Cauchy problem:
\begin{equation}
\label{06111}
    \begin{cases}
   \rho\, \partial_t u + (- \Delta)^{\frac{\sigma}{2}}\left[ u^m\right] = 0  & x\in \Rr^N, \quad t>0\\
     u=u_0 & x\in \Rr^N, \quad t=0.
    \end{cases}
\end{equation}
The parameter $m$ is greater or equal to $1$, and we will take later $m=1/\alpha$.
\begin{definition}\label{defsol}
A \emph{solution} to problem \eqref{06111} is a function $u\geq 0$ such that:
\begin{itemize}
\item{} $u\in C([0,\infty); L^1_{\rho}(\Rr^N))$ and $u^m \in L^2_\text{loc}((0, \infty): \dot{H}^{\sigma/2}(\Rr^N))$;
\item for any $T>0$, $\psi\in C^{1}_0(\Rr^N \times (0, T))$ there holds
\begin{equation*}
\label{02016}
\int_0^T \int_{\Rr^N} \rho\, u\, \partial_t \psi \, \dd x\, \dd t - \int_0^T \int_{\Rr^N}(-\Delta)^{\sigma/4}(u^m)\,
(-\Delta)^{\sigma/4}\psi \, \dd x\, \dd t = 0;
\end{equation*}
\item $u(\cdot, 0)=u_0$ almost everywhere.
\end{itemize}
\end{definition}

The previous definitions can be adapted to consider problem \eqref{06111} in bounded domains. Let $R>0$, $u_0\in L^1_{\rho}(B_R)$ and consider the problem
\begin{equation}
\label{02013}
    \begin{cases}
   \rho\, \partial_t u + (- \Delta)^{\frac{\sigma}{2}}\left[ u^m\right] = 0  & x\in B_R, \, t>0,\\
     u=0 & x\in \partial B_R, \, t>0,\\
     u=u_0 &  x\in B_R, \, t=0.
     \end{cases}
\end{equation}

\begin{definition}
\label{02017}
A \emph{solution} to problem \eqref{02013} is a function $u\geq 0$ such that:
\begin{itemize}
\item{} $u\in C([0,\infty); L^1_{\rho}(B_R))$ and $u^m \in L^2_\text{loc}((0, \infty): \dot{H}^{\sigma/2}(B_R))$;
\item for any $T>0, \psi\in C^1_0(B_R \times (0,T))$ there holds
\begin{equation}
\label{02018}
\int_0^T \int_{B_R} \rho\, u\, \partial_t \psi \,\dd x\, \dd t = \int_0^T \int_{B_R}(-\Delta)^{\sigma/4}u^m\,
(-\Delta)^{\sigma/4}\psi \, \dd x\, \dd t;
\end{equation}
\item $u(\cdot, 0)=u_0$ almost everywhere in $B_R$.
\end{itemize}
\end{definition}

\smallskip
Observe that comparison principles hold for problem \eqref{02013} (see \cite{PT2}). Moreover, the existence of the minimal solution to problem \eqref{06111} has been established in \cite{PT2}, together with some uniqueness results, among which we recall for later use the following:
\begin{proposition}\label{tunc} Let $N\geq 2$.
Let assumptions \eqref{A0}, \eqref{A2b} be satisfied. Moreover, suppose that $\rho>0$ in $\Rr^N$, $ u_0\in L^{\infty}(\Rr^N) \cap L^+_\rho(\Rr^N)$,
$m\ge 1$. Then there exists the minimal nonnegative solution $\underline u$ to problem \eqref{06111}; moreover,
\begin{equation}
\label{ea65}
\begin{array}{c}
\displaystyle \int_0^t  \big(\underline u(x,s) \big)^m \dd s  \leq C |x|^{\sigma-\nu-\frac N r} \\ \textrm{ for almost every } x\in \Rr^N\setminus B_{\bar R}\, (\bar R>0), \, t>0,  \\
\text{ for some $C>0$,  with $\nu, r$ as in Lemma \ref{lemmaRub2}}.
\end{array}
\end{equation}
Furthermore,  suppose that $u$ is a solution to problem
\eqref{06111} such that, for any $T>0$,
\begin{equation}
\label{ea65b}
\begin{array}{c}
\displaystyle \int_0^t  \big(u(x,s) \big)^m \dd s  \leq C_T
|x|^{\sigma-\nu-\frac N r} \\ \textrm{ for almost every }
x\in \Rr^N\setminus B_{\bar R}\, (\bar R>0), \, t\in (0,T),  \\
\text{ for some $C_T>0$,  with $\nu, r$ as in Lemma
\ref{lemmaRub2}}.
\end{array}
\end{equation}
Then $u\equiv \underline{u}$ .
\end{proposition}
\begin{proof}
See \cite[Theorem 5.9]{PT2}
\end{proof}

\subsection{Uniqueness of solutions for the elliptic problem}

\begin{lemma}
\label{lemma**}
Let $u_1$ and $u_2$ respectively a subsolution and a supersolution of \eqref{19021}.
Then there exists $u$ solution to \eqref{19021} such that $u_1 \leq u \leq u_2$ in $\Rr^N$.
\end{lemma}

\begin{proof}
Thanks to Lemma \ref{lemma*}, we can apply the standard technique of monotone iteration in the whole $\Rr^N$, and get the conclusion (note that the same argument has been
applied in the proof of \cite[Theorem 2]{BK}).
\end{proof}

\begin{lemma}
\label{01031}
Let $\rho_1$ and $\rho_2$ satisfying \eqref{A0}--(i), \eqref{A2b}, and assume $\rho_1\leq \rho_2$. Let $N>2\sigma.$ Then, for any $u_1$ bounded solution to
\[
-(\Delta)^\frac{\sigma}{2}u_1 = \rho_1\, u_1^\alpha \qquad\text{in }\Rr^N
\]
there exists $u_2$ bounded solution to
\begin{equation}
\label{01032}
(-\Delta)^\frac{\sigma}{2}u_2 = \rho_2\, u_2^\alpha \qquad\text{in }\Rr^N
\end{equation}
such that
\begin{equation}\label{e17041}
u_2(x) \leq C |x|^{\sigma - \nu - \frac{N}{r}}, \qquad u_1\leq u_2 \text{ in }\Rr^N
\end{equation}
for some $C>0$, with $\nu$ and $r$ as in Lemma \ref{lemmaRub2}.
\end{lemma}

\begin{proof}
Set
\[
\tilde C= \left(\|u_1\|_{L^\infty(\Rr^N)}\right)^\alpha.
\]
Then
\[
(-\Delta)^\frac{\sigma}{2}u_1 \leq \rho_2\, u_1^\alpha \leq \tilde C\rho_2 \qquad\text{in }\Rr^N.
\]
The function $V:= \bar{C}(K^\sigma * \rho_2)$ satisfies, for $\bar{C}>\tilde C$ sufficiently large,
\[
(-\Delta)^\frac{\sigma}{2}V  = \bar{C} \rho_2 \geq \rho_2  V^\alpha \qquad\text{in }\Rr^N.
\]
Thus $u_1$ and $V$ are respectively a subsolution and a supersolution of the same problem:
\[
(-\Delta)^\frac{\sigma}{2}U = \bar{C} \rho_2 \qquad\text{in }\Rr^N.
\]
By Lemma \ref{lemma*}, $u_1\leq V$ in $\Rr^N$.  Hence from Lemma \ref{lemma**} there exists a solution $u_2$ to problem \eqref{01032} such that
\[ u_1 \leq u_2 \leq V\quad \textrm{in}\;\; \Rr^N\,.\]
So, from Lemma \ref{lemmaRub2}  we get \eqref{e17041}. This completes the proof.
\end{proof}

We establish first uniqueness under the stronger assumption that $\rho>0$:

\begin{proposition}
\label{02034}
Assume \eqref{A0}, \eqref{A2b}; let $N>2\sigma$. Suppose further that $\rho>0$.
Let $\underline{u}$ be the minimal bounded solution to problem \eqref{19021} provided by Theorem \ref{190217}.
Let $u$ be any other bounded solution to problem \eqref{19021} such that
\[
u (x) \leq C |x|^{\sigma - \nu - \frac{N}{r}},
\]
for some $C>0$, with $r$ and $\nu$ as in Lemma \ref{lemmaRub2}. Then $\underline{u}=u$ in $\Rr^N$.
\end{proposition}

\begin{proof}
Set $m:= 1/\alpha$ and
\[
C_m:= (m-1)^{-\frac{1}{m-1}.}
\]
Let $v_R(x,t)$ be the solution to
\begin{equation}
\label{02035}
    \begin{cases}
     \displaystyle\rho\, \frac{\partial v_R}{\partial t} + (- \Delta)^{\frac{\sigma}{2}}\left[ v_R^m\right] = 0  & x\in B_R, \, t>0,\\
     v_R=0 & x\in \partial B_R, \, t>0,\\
          v_R(x,0)= C_m u^\frac{1}{m} &  x\in B_R.
     \end{cases}
\end{equation}

Observe that the function
\[
\tilde{u}(x,t) :=  \frac{C_m}{(t+1)^{\frac{1}{m-1}}} u^\frac{1}{m}(x)
\]
solves
\[
\rho\, \frac{\partial \tilde{u}}{\partial t} + (- \Delta)^{\frac{\sigma}{2}}\left[ \tilde{u}^m\right] = 0, \qquad\text{in }\Rr^N\times (0, \infty).
\]
Moreover, $\tilde u$ is a supersolution to problem \eqref{02035}. Thus, by comparison principles,
\begin{equation}
\label{04031}
v_R \leq \tilde{u} \qquad\text{in }B_R\times (0, \infty).
\end{equation}

Notice that for any $R>0$
\[
\D{ess}\inf_{B_R} \underline{u} >0.
\]
Then we can select $\tau_R>0$ such that
\[
\frac{\underline{u}^\frac{1}{m}}{\tau_R^{\frac{1}{m-1}}}> {u}^\frac{1}{m} \qquad\text{in }B_R.
\]
We have
\[
\underline{\check{u}} (x,t) :=
\frac{C_m\underline{u}^\frac{1}{m}}{(t+\tau_R)^{\frac{1}{m-1}}} \leq
\frac{C_m\underline{u}^\frac{1}{m}}{t^{\frac{1}{m-1}}} %{u}^\frac{1}{m}
=: \underline{\tilde{u}} (x,t) \qquad\text{in }B_R \times (0, \infty);
\]
moreover $\underline{\check{u}}$ is a supersolution to \eqref{02035} thus, by comparison principles we get
\begin{equation}
\label{04032}
v_R \leq \underline{\check{u}} \leq  \underline{\tilde{u}} \qquad\text{in }B_R\times (0, \infty).
\end{equation}

Now, by results in \cite{PT2} there exists the limit
\[
v_\infty := \lim_{R\to \infty} v_R;
\]
the function $v_\infty$ solves
\begin{equation}
\label{04033}
    \begin{cases}
     \displaystyle\rho\, \frac{\partial v_\infty}{\partial t} + (- \Delta)^{\frac{\sigma}{2}}\left[ v_\infty^m\right] = 0  & x\in \Rr^N, \, t>0,\\
     v_R(x,0)= C_m u^\frac{1}{m} &  x\in \Rr^N,
     \end{cases}
\end{equation}
and satisfies the inequality
\begin{equation}
\label{04034}
v_\infty^m (x, t) \leq C |x|^{\sigma - \nu- \frac{N}{r}}\quad (x\in \Rr^N, t>0)\,
\end{equation}
for some $C>0$, with $\nu$ and $r$ as in Lemma \ref{lemmaRub2}.
Then, by passing to the limit as $R\to \infty$ in \eqref{04031},
\[
v_\infty \leq \tilde{u} \qquad\text{in }\Rr^N\times (0, \infty).
\]
Notice that, as well as $v_\infty$, the function $\tilde{u}$ solves  \eqref{04033} and satisfies the inequality \eqref{04034}. Then, by Proposition \ref{tunc}
\[
v_\infty = \tilde{u} \qquad\text{in }\Rr^N\times (0, \infty).
\]
Passing to the limit as $R\to \infty$ in \eqref{04032}, we obtain
\[
v_\infty \leq \underline{\tilde{u}} \qquad\text{in }\Rr^N\times (0, \infty)
\]
which in turns entails
\[
\frac{u^\frac{1}{m}}{\underline{u}^\frac{1}{m}} \leq \frac{(t+1)^\frac{1}{m-1}}{t^\frac{1}{m-1}}.
\]
As $t\to + \infty$ we get
\[
u^\frac{1}{m} \leq \underline{u}^\frac{1}{m} \qquad\text{in }\Rr^N.
\]
Since $\underline u$ is minimal it follows that $u = \underline{u}$.
\end{proof}

We discuss now the general case in which $\rho\geq 0$.

\begin{theorem}\label{tuniell}
Assume \eqref{A0}, \eqref{A2b}; let $N>2\sigma$.
Let $\underline{u}$  and $u$ be as in Proposition \ref{02034}. Then $\underline{u}=u$ in $\Rr^N$.
\end{theorem}

\begin{proof}
Let $h\in C^\infty(\Rr^N) \cap L^\infty(\Rr^N) \cap L^1(\Rr^N)$, $h> 0$, and define for any $\epsilon>0$,
\[
\rho_\epsilon := \rho + \epsilon h.
\]
By Lemma \ref{01031} there exists $u_\epsilon$ solving
\begin{equation}
\label{04037}
-(\Delta)^\frac{\sigma}{2}u_\epsilon = \rho_\epsilon\, u_\epsilon^\alpha \qquad\text{in }\Rr^N
\end{equation}
and verifying the following inequalities  in $\Rr^N$:
\begin{gather}
u_\epsilon(x) \leq C |x|^{\sigma - \nu - \frac{N}{r}}, \nonumber\\
u\leq u_\epsilon, \label{04036}
\end{gather}
for some $C>0$, with $\nu$ and $r$ as in Lemma \ref{lemmaRub2}.
Thanks to Proposition \ref{02034} such $u_\epsilon$ is the unique
solution of \eqref{04037}. Let $u_{\epsilon, R}$ and $u_R$ be the
positive solutions to
\begin{equation}
\label{04038}
\begin{cases}
(-\Delta)^\frac{\sigma}{2}u_{\epsilon, R} = \rho_\epsilon \, u_{\epsilon, R}^\alpha &\text{in } B_R\\
u_{\epsilon, R}=0 & \text{in }\partial B_R,
\end{cases}
\end{equation}
and
\begin{equation}
\label{04039}
\begin{cases}
(-\Delta)^\frac{\sigma}{2}u_{R} = \rho\, u_R^\alpha &\text{in } B_R\\
u_{R}=0 & \text{in }\partial B_R.
\end{cases}
\end{equation}
So, for any $\varphi\in C^\infty_0(B_R), \varphi =0$ on $\partial
B_R$,
\begin{align}
\int_{ B_R} \rho_\epsilon\, u_{\epsilon, R}^\alpha\, \varphi(x,0)
\,\dd x &=  \int_{B_R} (-\Delta)^{\frac{\sigma}4}\varphi
(-\Delta)^{\frac{\sigma}4} u_{\epsilon, R}\,\dd x\, \dd y; \label{04039}\\
\int_{ B_R} \rho\, u_{R}^\alpha\, \varphi(x) \,\dd x &= \int_{B_R}
(-\Delta)^{\frac{\sigma}4}\varphi (-\Delta)^{\frac{\sigma}4} u_{R}
\,\dd x\, . \label{040310}
\end{align}
It is easily seen that \eqref{04039} holds true with $\varphi=
u_R$, while  \eqref{040310} holds true with  $\varphi=
u_{\epsilon,R}$; so, we obtain:
\begin{multline}
\int_{ B_R} \rho_\epsilon\, u_{\epsilon, R}^\alpha\, u_R(x) \,\dd x =
   \int_{B_R} (-\Delta)^{\frac{\sigma}4} u_{\epsilon,R}(-\Delta)^{\frac{\sigma}4} u_{R} \,\dd x\,\\
=\int_{ B_R} \rho\, u_{R}^\alpha\, u_{\epsilon, R}(x) \,\dd x.
\end{multline}
Hence,
\begin{align*}
\int_{ B_R} \rho\, u_{\epsilon, R} u_R^\alpha (u_{\epsilon, R}^{1-\alpha} - u_R^{1-\alpha})\,\dd x
& = \int_{ B_R} \rho \bigl[ u_{\epsilon, R} \, u_R^\alpha  - u_{\epsilon, R}^{\alpha} \, u_R \bigr]\dd x\\
& = \int_{ B_R} \bigl[ \rho_\epsilon  \, u_{\epsilon, R}^\alpha \, u_R  - \rho \, u_{\epsilon, R}^{\alpha} \, u_R \bigr]\dd x\\
& = \int_{ B_R} (\rho_\epsilon - \rho) \, u_{\epsilon, R}^\alpha \, u_R \, \dd x\\
& \leq \int_{B_R} \epsilon\, h\, u_{\epsilon, R}^\alpha \, u_R \, \dd x \leq C\epsilon \|h\|_{L^1(\Rr^N)} \leq C\epsilon
\end{align*}
for some $C>0$ independent of $R$. Passing to the limit as $R\to \infty$ and taking into account \eqref{04036} we get
\begin{multline}
\label{040311}
\int_{\Rr^N} \rho\, u^\alpha\, \underline{u}^\alpha \, (u^{1-\alpha} - \underline{u}^{1-\alpha})\,\dd x \leq
\int_{\Rr^N} \rho\, u_\epsilon^\alpha\, \underline{u}^\alpha \, (u_\epsilon^{1-\alpha} - \underline{u}^{1-\alpha})\,\dd x \\
\leq \lim_{R\to \infty} \int_{ B_R} \rho\, u_{\epsilon, R} \, u_R^\alpha (u_{\epsilon, R}^{1-\alpha} - u_R^{1-\alpha})\,\dd x \leq C\epsilon.
\end{multline}
Since $u\geq \underline{u}$, by sending $\epsilon\to 0^+$ in \eqref{040311} we discover
\[
\int_{\Rr^N} \rho\, u^\alpha\, \underline{u}^\alpha \, (u^{1-\alpha} - \underline{u}^{1-\alpha})\,\dd x  =0.
\]
Hence $\rho u^\alpha = \rho \underline{u}^\alpha$ in $\Rr^N$, which implies
\[
(-\Delta)^\frac{\sigma}{2}(u - \underline{u}) = \rho\, (u^\alpha - \underline{u}^\alpha) = 0 \qquad\text{in }\Rr^N.
\]
By uniqueness of solutions for the linear problem (see Lemma \ref{27026}), we conclude that $u=\underline{u}$ in $\Rr^N$.
\end{proof}

\bibliographystyle{plain}

\end{document}